\documentclass{article}
\usepackage{graphicx} 
\usepackage{tikz}
\title{Quantum Plane}

\usepackage[english]{babel}

\usepackage[letterpaper,top=2cm,bottom=2cm,left=3cm,right=3cm,marginparwidth=1.75cm]{geometry}

\usepackage{amsmath, amssymb, latexsym}
\usepackage{graphicx}
\usepackage[colorlinks=true, allcolors=blue]{hyperref}
\usepackage{amsfonts}
\usepackage{amsthm}
\usepackage{color}
\usepackage{indentfirst}
\usepackage{comment}

\usepackage{enumerate,xcolor}
\usepackage{amsrefs}

\newtheorem{theorem}{Theorem}

\newtheorem{proposition}[theorem]{Proposition}
\newtheorem{lemma}[theorem]{Lemma}
\newtheorem{example}[theorem]{Example}
\newtheorem{remark}[theorem]{Remark}

\DeclareMathOperator{\Der}{Der}
\DeclareMathOperator{\Derint}{Derint}
\DeclareMathOperator{\Aut}{Aut}
\DeclareMathOperator{\centro}{Z}
\DeclareMathOperator{\ad}{ad}
\DeclareMathOperator{\ord}{o}

\title{\textbf{\LARGE On Isotropy Groups of Quantum Plane}}
\author{ \small
ADRIANO DE SANTANA, RENE BALTAZAR, ROBSON VINCIGUERRA, WILIAN DE ARAUJO}
\date{\vspace{-5ex}}

\begin{document}
\Large{
\maketitle

\begin{abstract}
This paper investigates the isotropy groups of derivations on the Quantum Plane $\Bbbk_q[x, y]$, defined by the relation \(yx = qxy\), where \(q \in \Bbbk^*\), with $q^2\neq 1$. The main goal is to determine the automorphisms of the Quantum Plane that commutes with a fixed derivation \(\delta\). We describe conditions under which the isotropy group \(\text{Aut}_\delta(A)\) is trivial, finite, or infinite, depending on the structure of \(\delta\) and whether \(q\) is a root of unity: additionally, we present the structure of the group in the finite case. A key tool is the analysis of polynomial equations of the form \(\mu_1^a \mu_2^b = 1\), arising from monomials in the inner part of \(\delta\). We also make explicit which finite subgroups of $\Aut(\Bbbk_q[x, y])$ are isotropy groups of some derivation: either $q$ root of unity or not. Techniques from algebraic geometry, such as intersection multiplicity, are also employed in the classification of the finite case.

\end{abstract}

\section{Introduction}

Let $\delta$ be a derivation of a $\Bbbk$-algebra $R$ and $\Aut(R)$ the group of $\Bbbk$-automorphisms of $R$ with $\Bbbk$ an infinite field. Note that $Aut(R)$ acts by conjugation over the module of derivations of $R$, $\Der(R)$: given $\delta \in \Der(R)$ and $\rho \in \Aut(R)$, then $\rho \delta \rho^{-1} \in \Der(R)$. The \textit{isotropy} subgroup, with respect to this group action, is defined
as
$$\Aut_\delta(R):=\{ \rho \in \Aut(R) \ \vert \ \rho \delta = \delta\rho \}. $$

The behavior of an isotropy subgroup is far from being completely understood. In recent years, there has been a notable increase in research focused on the isotropy group. This growing interest is reflected in the contributions of researchers such as: I. Pan, L. Mendes, D. Levcovitz, L. Bertoncello, R. Baltazar, D. Yan, Y. Huang, M. Veloso, N. Dasgupta, A. Lahiri, A. Rittatore, S. Kour, and H. Rewri (see \cite{Kour, BaltPan, Rittatore, Dasgupta, BaltaVeloso, HuYan, Balta, LevBert, PanMendes}). On this topic in the case of a quantized Weyl algebra, see \cite{SBVA}.

The Quantum Plane, $\Bbbk_q[x,y]$, is defined as the quotient of the free associative $\Bbbk$-algebra $ \Bbbk\langle x, y \rangle $, generated by two variables $ x $ and $ y $ , by the two-sided ideal generated by the relation $yx - qxy$, where $ q \in \Bbbk^*$. When $ q = 1 $, the relation becomes commutative, recovering the classical polynomial ring $\Bbbk[x, y]$.

It is quite common for different sets of generators and relations to define the same intrinsic structure, which is why algebras are typically classified up to isomorphism.
The Quantum Plane arises from an interesting classical result: an Ore extension over a polynomial algebra $\Bbbk[x]$ is either a Quantum Plane, a Quantum Weyl algebra, or an infinite-dimensional associative algebra $A_h$ generated by elements $x, y$, that satisfy $yx-xy = h$, where $h \in \Bbbk[x]$ (for more details, see Lemma 1.2. \cite{Samuel} and \cite{SBVA}). For example, it is already well known that considering quantization of the Jordan Plane we obtain an algebra isomorphic to the Quantum Plane, see \cite[Example 3.1]{Samuel24}.

Relying on the characterization of derivations and automorphisms in the Quantum Plane provided by Alev and Chamarie \cite[Theorem 1.2 and Proposition 1.4.4]{Alev}, this work investigates, throughout sections 3, 4 and 5, the following fundamental questions:

\begin{itemize}
    \item  Fixed $\delta \in \Der (\Bbbk_q[x,y])$, is it possible to make explicit $\Aut_\delta(\Bbbk_q[x,y]$)?
    \item Does the class of isotropy groups of $\Bbbk_q[x,y]$ includes, up to isomorphism, all finite groups of $\Aut(\Bbbk_q[x,y])$?
    \item Are there necessary and sufficient conditions for the isotropy of a fixed derivation to be trivial?
\end{itemize}

In the following section, this analysis is complemented by a geometric approach, where we applied intersection multiplicity to the curves defined by $ x^a y^b = 1 $ and $ x^c y^d = 1$, confirming previous results and offering a geometric interpretation of the number of affine intersection distinct points, $|ad - bc|$.


\section{Generalities}

Let \( \Bbbk \) be a field and let \( R \) be a \( \Bbbk \)-algebra of dimension two generated by the elements \( x \) and \( y \). Let \( \delta: R \to R \) be a derivation on \( R \) defined by the polynomials
\[
\delta(x) = \sum_{i,j} a_{ij} x^i y^j
\quad
\mbox{and}
\quad
\delta(y) = \sum_{i,j} b_{ij} x^i y^j.
\]

Let \( \rho \in \mathrm{Aut}(R) \) be an automorphism defined by \( R \)-linearity from the identities
\[
\rho(x) = \mu_1 x, \quad \rho(y) = \mu_2 y,
\]
where \( \mu_1, \mu_2 \in \Bbbk^* \).

Assume now that \( \rho \in \mathrm{Aut}_\delta(R) \); then, $\rho \delta - \delta \rho = 0$. Applying this commutativity condition to the generator \( x \), we obtain:
\begin{equation}
\sum_{i,j} a_{ij} \mu_1 \left( \mu_1^{i-1} \mu_2^j - 1 \right) x^i y^j = 0. \label{eq:com x}
\end{equation}
Similarly, applying it to the generator \( y \), we obtain:
\begin{equation}
\sum_{i,j} b_{ij} \mu_2 \left( \mu_1^i \mu_2^{j-1} - 1 \right) x^i y^j = 0. \label{eq:com y}
\end{equation}

By Equation \eqref{eq:com x}, we must have that \( \mu_1^{i-1} \mu_2^j = 1 \) for every pair \( (i,j) \) such that \( a_{ij} \neq 0 \). Similarly, by Equation~\eqref{eq:com y}, we must have \( \mu_1^i \mu_2^{j-1} = 1 \) for every pair \( (i,j) \) such that \( b_{ij} \neq 0 \). Both of these conditions reduce to requiring that \( \mu_1 \) and \( \mu_2 \) satisfy equations of the form
\begin{equation}
    \mu_1^m \mu_2^n - 1 = 0, \label{eq:MuMu}
\end{equation}
for suitable integers \( m \geq -1 \) and \( n \geq -1 \) depending on the derivation \( \delta \).

\begin{remark}\label{remark0} We have the following known results:

If $R=\Bbbk_q[x,y]$ is the Quantum Plane, with \( q \neq \pm 1 \), then any automorphism \( \rho \in \mathrm{Aut}(R) \) is such that \( \rho(x) = \mu_1 x \) and \( \rho(y) = \mu_2 y \), with \( \mu_1, \mu_2 \in \Bbbk^* \) \cite[Proposition 1.4.4]{Alev}. Thus, \( \Aut(A) \cong \Bbbk^* \oplus \Bbbk^*\).

If $R$ is the quantized Weyl algebra, according to~\cite[Theorem B]{vivas}, a map \( \rho: R \to R \) given by \( \rho(x) = \mu_1 x \) and \( \rho(y) = \mu_2 y \) defines an automorphism if and only if \( \mu_1 = \mu_2^{-1} \).  Thus, \( \Aut(R) \cong \Bbbk^* \).

In the case $R=\Bbbk_{-1}[x,y]$, an automorphism \( \rho \) is given by \( \rho(x) = \mu_1 \sigma(x) \) and \( \rho(y) = \mu_2 \sigma(y) \), where \( \sigma \) is a permutation of the pair \( (x,y) \). Thus, \( \Aut(R) \cong \mathbb{Z}_2 \ltimes (\Bbbk^* \oplus \Bbbk^*) \) (\cite[Proposition 1.4.4]{Alev}).

\end{remark}


In what follows, we obtain a description of the isotropy groups when $R=\Bbbk_q[x,y]$ is the Quantum Plane, with \( q \neq \pm 1 \).

\section{Equations for the Quantum Plane with $q \neq \pm 1$}

Let \( A = \Bbbk_q[x,y] \) be the Quantum Plane, with \( q \neq \pm 1 \). It is known that 
\begin{equation}
    \label{eq:estrutura_der_a}
    \Der(A) = \Derint(A) \oplus \centro(A)D_x \oplus \centro(A)D_y
\end{equation}
where \( \Derint(A) \) is the group of inner derivations of \( A \), \( \centro(A) \) is the center of the algebra, \( D_x = x\partial_x \) and \( D_y = y\partial_y \) (see \cite[Theorem 1.2]{Alev}).

Let \( \rho \in \Aut(A) \), \( \rho(x) = \mu_1 x \) and \( \rho(y) = \mu_2 y \) (see Remark \ref{remark0}). Clearly, $\rho \in \Aut_{D_x}(A) \cap \Aut_{D_y}(A)$. Thus, if \( \delta = \ad_w + aD_x + bD_y \), then \( \rho\delta = \delta\rho \) if and only if \( \rho\ad_w = \ad_w\rho \); that is, \( \rho \in \Aut_\delta(A) \) if and only if \( \rho \) commutes with the inner part of \( \delta \).


If \( w = \sum c_{ij}x^i y^j \), then \( \rho \ad_w = \ad_w \rho \) if and only if
\begin{equation}
    \label{eq:mu_da_interna_x}
    \begin{array}{rcl}
        0 &=& (\rho \ad_w - \ad_w \rho)(x) \\
          &=& \sum c_{ij} \mu_1 (q^j - 1)(\mu_1^i \mu_2^j - 1)x^{i+1} y^j
    \end{array}
\end{equation}
and
\begin{equation}
    \label{eq:mu_da_interna_y}
    \begin{array}{rcl}
        0 &=& (\rho \ad_w - \ad_w \rho)(y) \\
          &=& \sum c_{ij} \mu_2 (1 - q^i)(\mu_1^i \mu_2^j - 1)x^i y^{j+1}.
    \end{array}
\end{equation}

We observe from Equations \eqref{eq:mu_da_interna_x} and \eqref{eq:mu_da_interna_y} that \( \rho \in \Aut_\delta(A) \) if and only if, for all \( c_{ij} \neq 0 \) such that \( q^j - 1 \neq 0 \) or \( 1 - q^i \neq 0 \), the constants \( \mu_1, \mu_2 \in \Bbbk^* \) satisfy \( \mu_1^i \mu_2^j - 1 = 0 \). In particular, if \( q \) is not a root of unity, then
$$
\Aut_\delta(A) = \{(\mu_1, \mu_2) \in \Bbbk^* \oplus \Bbbk^* : (\forall c_{ij} \neq 0)(\mu_1^i \mu_2^j - 1 = 0)\}.
$$

\textbf{Examples:} Suppose that \( q \) is not a root of unity and let \( \delta = \ad_w + aD_x + bD_y \). Then we have:
\begin{itemize}
    \item if \( w = x \), then \( \Aut_\delta(A) \cong \Bbbk^* \) (The same holds if \( w = y \));
    \item if \( w = x^m \), then \( \Aut_\delta(A) \subseteq \mathbb{Z}_m \oplus \Bbbk^* \) (closed immersion), with equality holding when \( \Bbbk \) is algebraically closed;
    \item if \( w = x^m + y^n \), then \( \Aut_\delta(A) \subseteq \mathbb{Z}_m \oplus \mathbb{Z}_n \) (closed immersion), with equality holding when \( \Bbbk \) is algebraically closed;
    \item if \( w = x^m y^n \) consists of a unique monomial, then \( \Aut_\delta(A) \) is infinite;
    \item if \( w \) is constant, then \( \Aut_\delta(A) = \Aut(A) \).
\end{itemize}

\begin{remark}
  Note that in the above examples we considered $\Bbbk$ to be an algebraically closed field to represent the isomorphic structure of the isotropy group. Even though the field is not algebraically closed, we can solve the Equations \eqref{eq:MuMu} in the algebraic closure and take the solutions that belong to the original field. The result will be a subgroup of what we would obtain in the algebraic closure. In what follows, in some circumstances, we ignore the case that the field is not algebraically closed in order to fully present the structure of the isotropy group.
\end{remark}

\begin{remark}\label{remarkcharac}
   If $q\neq \pm 1$ then $\Aut(\Bbbk_q[x,y]) \cong \Bbbk^* \times \Bbbk^* $ (a split torus). Hence, $\Aut_\delta(A)$ being a closed subgroup, it is the intersection of the kernels of some characters (see \cite[Chapter III]{Borel}). Thus, equations \eqref{eq:mu_da_interna_x} and \eqref{eq:mu_da_interna_y} give which characters one should take in order to obtain $\Aut_\delta(A)$. 
\end{remark}
\section{Conditions for Finiteness of
 $\Aut_\delta(A)$}

As observed previously, the elements of \( \Aut_\delta(A) \) can be seen as the set of points in the affine plane over \( \Bbbk \) that are intersections of polynomial equations of the form \( \mu_1^i \mu_2^j - 1 = 0 \), where these equations are determined by the internal part of the derivation \( \delta \) and by conditions on \( q \). Let \( \mu_1^{m_i} \mu_2^{n_i} - 1 = 0 \) with \( i = 1, 2, \dots, r \) be the set of equations defining \( \Aut_\delta(A) \). From classic algebraic geometry, when $\Bbbk$ is algebraically closed, we know that \( \Aut_\delta(A) \) is an infinite set if and only if all these polynomial equations have a common non-constant factor. However, as we will see in Proposition \ref{Prop2}, if these equations share a common factor, then they can actually be factored over the base field. In this case, we find that \( \Aut_\delta(A) \) is infinite if and only if $\Bbbk$ is infinite and all these polynomial equations admit a non-constant common factor.

\begin{lemma}
    \label{lem:mus}
    For each \( i = 1, \dots, r \), let \( m_i, n_i \in \mathbb{Z} \) such that, for all\linebreak \( i, j \in \{1, \dots, r\} \), $m_i n_j - m_j n_i = 0$. Then, if \( m = \gcd(m_1, \dots, m_r) \) and \( n = \gcd(n_1, \dots, n_r) \), we obtain $m_i/m = n_i/n$. 
\end{lemma}

\begin{proof}
 The result follows using decomposition on prime factors. 
\end{proof}

\begin{proposition}\label{Prop2}
    For \( i = 1, 2, \dots, r \), let \( m_1, m_2, \dots, m_r \) and \( n_1, n_2, \dots, n_r \) be non-negative integers. Then the equations \( \mu_1^{m_i} \mu_2^{n_i} - 1 = 0 \) have a non-constant common factor if and only if, for every pair \( 1 \le i, j \le r \), we have \( m_i n_j - m_j n_i = 0 \).
\end{proposition}

\begin{proof}
    First, suppose that there exists a pair \( 1 \le i, j \le r \) such that \( m_i n_j - m_j n_i \neq 0 \). We prove that the equations
\[
\mu_1^{m_i} \mu_2^{n_i} - 1 = 0 
\quad \text{and} \quad 
\mu_1^{m_j} \mu_2^{n_j} - 1 = 0
\]
do not have a non-constant common factor in $\Bbbk[\mu_1,\mu_2]$. 
Indeed, suppose by contradiction that they do, and consider them in 
$\overline{\Bbbk[\mu_2]}[\mu_1]$, where $\overline{\Bbbk[\mu_2]}$ is the 
algebraic closure of the field of fractions of $\Bbbk[\mu_2]$. 
Then we can write $\mu_1^{m_i}=\mu_2^{-n_i}$ and $\mu_1^{m_j}=\mu_2^{-n_j}$. The solutions of these equations are respectively
\[
\mu_1 = \xi_{m_i}^t \sqrt[m_i]{\mu_2^{-n_i}} 
\quad \text{with } t=1,\ldots,m_i-1,
\]
and
\[
\mu_1 = \xi_{m_j}^s \sqrt[m_j]{\mu_2^{-n_j}} 
\quad \text{with } s=1,\ldots,m_j-1.
\]

Since the equations share a non-constant component, there exists a common 
root in $\overline{\Bbbk[\mu_2]}$, that is, there exist $t$ and $s$ such that
\[
\xi_{m_i}^t \sqrt[m_i]{\mu_2^{-n_i}} 
= \xi_{m_j}^s \sqrt[m_j]{\mu_2^{-n_j}}.
\]
But this implies $m_i n_j - m_j n_i = 0$, a contradiction.

    Now, suppose that for every pair \( i, j \), \( m_i n_j - m_j n_i = 0 \), and let \( m = \gcd(m_1, \dots, m_r) \) and \( n = \gcd(n_1, \dots, n_r) \). By Lemma \ref{lem:mus}, we have that for every \( i = 1, \dots, r \), \( m_i / m = n_i / n \). Thus, if \( b_i = m_i / m = n_i / n \), it follows that
    \[
    \mu_1^{m_i} \mu_2^{n_i} - 1 = (\mu_1^m \mu_2^n - 1) \sum_{l=0}^{b_i} \mu_1^{m_i - lm} \mu_2^{n_i - ln},
    \]
    which means that the equations \( \mu_1^{m_i} \mu_2^{n_i} - 1 = 0 \) has \( \mu_1^m \mu_2^n - 1 = 0 \) as a non-constant common factor.
\end{proof}

\begin{proposition}
    Lets $\delta = a d_w + a D_x + b D_y$ be a derivation on $A$, with $w = \sum c_{ij}x^i y^j$. If $q$ is not a root of unity, then $\Aut_\delta(A)$ is finite if and only if there exist $c_{ij} \neq 0$ and $c_{rs} \neq 0$ such that $is - rj \neq 0$. If $q$ is a root of unity of order $n \neq 2$, then $\Aut_\delta(A)$ is finite if and only if there exist $c_{ij} \neq 0$ and $c_{rs} \neq 0$ such that $is - rj \neq 0$ and $(i,j), (r,s) \notin n\mathbb{Z} \times n\mathbb{Z}$.
 
\end{proposition}

\begin{proof}
    It follows from Proposition \ref{Prop2} and equations \ref{eq:com x} and \ref{eq:com y}.
\end{proof}

Let $G$ be a finite subgroup of $\Bbbk^* \oplus \Bbbk^*$. Then, by the structure theorem for finite abelian groups,
$$
G \cong \mathbb{Z}_{n_1} \oplus \mathbb{Z}_{n_2} \oplus \ldots \oplus \mathbb{Z}_{n_r},
$$
where $n_{i+1} \mid n_i$. If $n_3 = n \neq 1$, then there exists a subgroup $H$ of $G$ such that $H \cong \mathbb{Z}_n \oplus \mathbb{Z}_n \oplus \mathbb{Z}_n$. Note that the equation $x^n - 1 = 0$ has $n^3$ solutions in $H$, but at most $n^2$ solutions in $\Bbbk^* \oplus \Bbbk^*$. Therefore, $G$ is not a subgroup of $\Bbbk^* \oplus \Bbbk^*$. Thus, we conclude that $G \cong \mathbb{Z}_{n_1} \oplus \mathbb{Z}_{n_2}$, with $n_2 \mid n_1$.

Note also that if $\Bbbk$ is algebraically closed, then for any $n_1, n_2 \in \mathbb{Z}$, there exists a subgroup $G < \Bbbk^* \oplus \Bbbk^*$ such that $G \cong \mathbb{Z}_{n_1} \oplus \mathbb{Z}_{n_2}$: it suffices to take $G = \langle \xi_{n_1}, \xi_{n_2} \rangle$, where $\xi_{n_i}$ is a primitive $n_i$-th root of unity.

We denote the order of an element $a$ in a group by $\ord(a)$.

\begin{theorem}\label{classisotropy} 
Let \( q \) be not a root of unity. Then the class of isotropy groups of \( A \) includes, up to isomorphism, all finite subgroups of \(\Aut(A)\).
\end{theorem}

\begin{proof}
Let \( G < \Aut(A) \) be finite. We will show that there exists \(\delta \in \Der(A)\) such that \(\Aut_\delta(A) \cong G\).

In fact, since \( G \subseteq \Aut(A) \cong \Bbbk^* \oplus \Bbbk^* \), there exist \( n_1, n_2 \in \mathbb{Z} \) such that \( n_2 \mid n_1 \) and \( G \cong \mathbb{Z}_{n_1} \oplus \mathbb{Z}_{n_2} \).
Hence, there exists \( x \in G \) such that \(\ord(x) = n_1\).
As \( G \subset \Bbbk^* \oplus \Bbbk^* \), there exist \( a,b \in \Bbbk^* \) such that \( x = (a,b) \) and \( a^{n_1} = 1 = b^{n_1} \).
Let \(\ord(a) = r\) and \(\ord(b) = s\).
Factoring into primes, we have 
\[
r = p_1^{r_1} p_2^{r_2} \ldots p_m^{r_m}, \quad \mbox{and} \quad s = p_1^{s_1} p_2^{s_2} \ldots p_m^{s_m}.
\]
Let \( d = \gcd(r,s) = p_1^{d_1} p_2^{d_2} \ldots p_m^{d_m} \), where \( d_i = \min(r_i, s_i) \). Moreover, for each \( i \), define
$$f_i = \begin{cases}
    r_i,    & \text {if } r_i=\max \{r_i,s_i\}\hphantom{\text{ and }  r_i \neq s_i}\\
    0,      & \text {otherwise,}
\end{cases}$$
\text {and }
$$g_i = \begin{cases}
    s_i,    & \text {if } s_i=\max \{r_i,s_i\} \text{ and }  r_i \neq s_i\\
    0,      & \text {otherwise.}
\end{cases}
$$
Let
\[
r' = p_1^{f_1} p_2^{f_2} \ldots p_m^{f_m} \quad \text{and} \quad s' = p_1^{g_1} p_2^{g_2} \ldots p_m^{g_m}.
\]
Notice that \(\gcd(r', s') = 1\) and \( r's' = d \cdot \frac{r}{d} \cdot \frac{s}{d} \).
Let \( a' = a^{\frac{r}{r'}} \) and \( b' = b^{\frac{s}{s'}} \), so \(\ord(a') = r'\) and \(\ord(b') = s'\).
Notice that  \(\ord(a' b') = r's'\).
Since 
$x^{r's'} = (a^{r \frac{s}{d}}, b^{s \frac{r}{d}}) = (1,1)$,
we have \( n_1 \mid r's' \).
Moreover,
$(a' b')^{n_1} = (a^{\frac{r }{r'}} b^{\frac{s }{s'}})^{n_1} = 1$,
so \( r's' \mid n_1 \).
We conclude that \( n_1 = r's' \) and hence \(\xi_{n_1} = a' b'\) is a primitive \( n_1 \)-th root of unity.

Let \( w = x^{n_1} + y^{n_2} \) and \(\delta = \ad_w\).
Since \( n_2 \mid n_1 \), there exists an integer \( p \) such that \( n_1 = n_2 p \)%
    , where $\xi_{n_2}=\xi_{n1}^p$.
Therefore,
    since $q^{n_1}\neq 1$, we obtain, by Equations \eqref{eq:mu_da_interna_x} and \eqref{eq:mu_da_interna_y}, 
\begin{eqnarray*}
    \Aut_\delta(A) 
        &=&\{(\mu_1,\mu_2)\in\Bbbk^*\otimes\Bbbk^*:\mu_1^{n_1}-1=0,\mu_2^{n_2}-1=0 \}\\
        &=&\{(\xi_{n_1}^i,\xi_{n_1}^{jp})\in\Bbbk^*\otimes\Bbbk^*:i,j\in\mathbb Z \}\\
        &\cong& \mathbb{Z}_{n_1} \oplus \mathbb{Z}_{n_2}\\
        &\cong& G.
\end{eqnarray*}
\end{proof}

\begin{remark}
Assume, according to the hypotheses of Theorem \ref{classisotropy}, that $\Bbbk$ is algebraically closed. Then, for any positive integers $m,n$, the group $\Bbbk^{\ast} \oplus \Bbbk^{\ast}$ contains as a subgroup $H_{n} \oplus H_{m}$, where $H_{r}$ denotes the group of all $r$-th roots of unity in $\Bbbk$. By Theorem \ref{classisotropy}, there exists a derivation $\delta$ such that $\Aut_\delta(A) \cong H_{n} \oplus H_{m} \cong G$.
\end{remark}

\begin{proposition}\label{Proq}
Let \( q \) be a primitive \( p \)-th root of unity with \( p \notin \{-1, 1\} \). Then, for any nonzero integers \( r \) and \( s \), there does not exist \(\delta \in \Der(A)\) such that \(\Aut_\delta(A) \cong \mathbb{Z}_{p r} \oplus \mathbb{Z}_{p s}.\)
\end{proposition}

\begin{proof}
Suppose there exists \(\delta \in \Der(A)\) such that \(\mathbb{Z}_{p r} \oplus \mathbb{Z}_{p s} \cong \Aut_\delta(A)\). By Equations \eqref{eq:mu_da_interna_x} and \eqref{eq:mu_da_interna_y}, there exist positive integers \( m_i, n_i \) for \( i=1,\ldots,l \) such that \(\rho \in \Aut_\delta(A)\), with \(\rho(x) = \mu_1 x\) and \(\rho(y) = \mu_2 y\) if and only if
$
\mu_1^{m_i} \mu_2^{n_i} - 1 = 0.
$

Note that \(\mathbb{Z}_p \oplus \mathbb{Z}_p < \mathbb{Z}_{p r} \oplus \mathbb{Z}_{p s}\), and also that the only subgroup of \( \Bbbk^* \oplus \Bbbk^* \) isomorphic to \(\mathbb{Z}_p \oplus \mathbb{Z}_p\) is 
$G = \{ (\xi_p^i, \xi_p^j) : i,j \in \mathbb{Z} \}.$
Indeed, the elements \( x \in \mathbb{Z}_p \oplus \mathbb{Z}_p \) satisfy \( x^p - 1 = 0 \); moreover, \(|\mathbb{Z}_p \oplus \mathbb{Z}_p| = p^2\), and in \( \Bbbk^* \oplus \Bbbk^* \), the set of elements satisfying \( x^p - 1 = 0 \) is exactly \( G \), which is also a group of \( p^2 \) elements.

Note that if 
\[
H_k = \{ (\mu_1, \mu_2) \in \Bbbk^* \oplus \Bbbk^* : \mu_1^{m_k} \mu_2^{n_k} - 1 = 0 \},
\]
then \(H_k\) contains a subgroup \(G_k \cong G\). Since \(G\) is unique, we have \( G < H_k \).
Thus, for every pair of integers \( i, j \), we have
$
\xi_p^{i m_k} \xi_p^{j n_k} - 1 = 0.$

Taking \( i=0 \) and \( j=1 \), we get \(\xi_p^{n_k} = 1\), which implies that \( n_k \) is a multiple of \( p \). On the other hand, for \( i=1 \) and \( j=0 \), we get \(\xi_p^{m_k} = 1\), and so \( m_k \) is also a multiple of \( p \).
We get a contradiction, since by Equations \eqref{eq:mu_da_interna_x} and \eqref{eq:mu_da_interna_y}, for \(\mu_1^{m_k} \mu_2^{n_k} - 1 = 0\) to be a condition defining \(\Aut_\delta(A)\), we must have either \( q^{m_k} - 1 \neq 0 \) or \( q^{n_k} - 1 \neq 0 \), which does not occur.
\end{proof}

\begin{remark}
    The class of isotropy groups over $A$ does not necessarily includes all infinite subgroups of $\Bbbk^*\oplus \Bbbk^*$. For example, if $\alpha$ is a transcendental element over $\Bbbk^*$, the subgroup $\langle(\alpha,1)\rangle$ is not obtained as a group $\Aut_\delta(A)$ (see Remark \ref{remarkcharac} about characters). 
\end{remark}

\begin{remark} Note that, at the end of the proof of Theorem~\ref{classisotropy}, the assumption that $q$ is not a root of unity is only used to conclude that $q^{n_1}\neq 1$.
 Assuming that $\Bbbk$ is algebraically closed, we can combine this observation with Proposition~\ref{Proq} in order to distinguish two quantum planes $\Bbbk_q[x,y]$ and $\Bbbk_{q_1}[x,y]$ whenever there exists $n$ such that $q^n \neq 1$ and $q_1^n = 1$.  
Indeed, the condition $q^n \neq 1$ implies the existence of $\delta \in \Der(\Bbbk_q[x,y])$ such that
$\Aut_\delta(\Bbbk_q[x,y]) \cong \mathbb{Z}_n \times \mathbb{Z}_n$,
while $q_1^n = 1$ implies that $\mathbb{Z}_n \times \mathbb{Z}_n$ is not isomorphic to any isotropy group of derivations in $\Der(\Bbbk_{q_1}[x,y])$.  
Moreover, if $\rho:\Bbbk_q[x,y] \to \Bbbk_{q_1}[x,y]$ is an isomorphism, then the map
$\rho^*: \Der(\Bbbk_q[x,y]) \to \Der(\Bbbk_{q_1}[x,y])$, defined by $\rho^*(\delta) = \rho \circ \delta \circ \rho^{-1}$,
is a bijection such that
$\Aut_\delta(\Bbbk_q[x,y]) \cong \Aut_{\rho^*(\delta)}(\Bbbk_{q_1}[x,y])$.
By contraposition, if there exists $n$ such that $q^n \neq 1$ and $q_1^n = 1$, it follows that $\Bbbk_q[x,y] \not\cong \Bbbk_{q_1}[x,y]$.
Although the complete classification of quantum planes is already known, in particular, by \cite[Theorem A]{vivas} we have $\Bbbk_q[x,y] \cong \Bbbk_{q_1}[x,y]$ if and only if
$q \in \{q_1, q_1^{-1}\}$.
However, the isotropy group approach provides a new perspective on the characterization of quantum planes, which may prove useful in the study of other classes of algebras for which the classification problems remains open.

\end{remark}
     
In the following results, we extract necessary and sufficient conditions to determine the structure of these isotropy groups. In particular, we demonstrate that it is possible to determine explicitly when this group is trivial.

\section{Structure of the Isotropy Group}

Let $\Bbbk$ be an algebraically closed field and let $a, b, c, d$ be positive integers such that $ad - bc \neq 0$. In this section, we determine the group structure of the subset $G \subseteq \Bbbk^* \oplus \Bbbk^*$ of pairs $(x, y)$ that satisfy the system of equations
\begin{equation}
    \label{eq:grupooriginal}
    \left\{\begin{array}{rcl}
        x^ay^b-1&=&0\\
        x^cy^d-1&=&0
    \end{array}\right.
\end{equation}

To this end, let $k = \gcd(a, b, c, d)$, and define $a_0 = a/k$, $b_0 = b/k$, $c_0 = c/k$, and $d_0 = d/k$. Also let $r = \gcd(a_0, c_0)$, $s = \gcd(b_0, d_0)$, and define $a_1 = a_0/r$, $c_1 = c_0/r$, $b_1 = b_0/s$, and $d_1 = d_0/s$.

Note that $\gcd(a_1, c_1) = 1$, so there exist integers $m$ and $n$ such that $a_1 m + c_1 n = 1$. Consider the following system:
\begin{equation}
    \label{eq:sistema_do_grupo}
    \left\{\begin{array}{rcl}
	x^{{a_1}}y^{{b_1}}&=1\\
	x^{{c_1}}y^{{d_1}}&=1
    \end{array}\right.  
\end{equation}

Raising both sides of the first equation to the power $m$, the second to the power $-n$, and isolating $y$, we obtain:

\begin{equation}
    \label{eq:sistema_do_grupo_2}
    \begin{cases}
    	x^{m{a_1}}&=y^{-m{b_1}}\\
    	x^{-n{c_1}}&=y^{n{d_1}}
\end{cases} 
\end{equation}

Dividing the first equation by the second gives $x = x^{m a_1 + n c_1} = y^{-m b_1 - n d_1}$. Substituting these into the equations of (3), we get:
\begin{equation}
    \label{eq:sistema_do_grupo_3}
    \begin{cases}
    	y^{(-m{b_1}-n{d_1}){a_1}+{b_1}}&=1\\
    	y^{(-m{b_1}-n{d_1}){c_1}+{d_1}}&=1
    \end{cases}
\end{equation}

Raising the first equation to the power $c_1$, the second to the power $a_1$, and dividing the second by the first, we get:
\begin{equation}
    \label{eq:isola_y}
    y^{{a_1}{d_1}-{b_1}{c_1}}=1
\end{equation}

This implies that there exists an integer $i$ such that $y = \xi_{a_1 d_1 - b_1 c_1}^i$, and consequently $x = \xi_{a_1 d_1 - b_1 c_1}^{-i(m b_1 + n d_1)}$. To verify that the pair
\begin{equation}
    \label{eq:solução_de_grupo2}
    (x_i,y_i)=\left(\xi_{{a_1}{d_1}-{b_1}{c_1}}^{-i(m{b_1}+n{d_1})},\xi_{{a_1}{d_1}-{b_1}{c_1}}^i\right)
\end{equation}
is a solution to the system \eqref{eq:grupooriginal}, we compute:
\begin{align*}
	x_i^{{a_1}}y_i^{{b_1}}
	&=\left(\xi_{{a_1}{d_1}-{b_1}{c_1}}^{-i(m{b_1}+n{d_1})}\right)^a\left(\xi_{{a_1}{d_1}-{b_1}{c_1}}^i\right)^b\\
	&=\xi_{{a_1}{d_1}-{b_1}{c_1}}^{i(-(m{a_1}+n{c_1}){b_1}+{b_1}+n({a_1}{d_1}+{b_1}{c_1}))}\\
	&=\xi_{{a_1}{d_1}-{b_1}{c_1}}^{in({a_1}{d_1}+{b_1}{c_1})}\\
	&=1.
\end{align*}
The same applies analogously for the second equation of the system.

Note that the system \eqref{eq:sistema_do_grupo} has $|a_1 d_1 - b_1 c_1|$ distinct solutions, and the group formed by these solutions is isomorphic to the cyclic group $\mathbb{Z}_{a_1 d_1 - b_1 c_1}$, since it is generated by the pair of elements in Equation \eqref{eq:solução_de_grupo2}.

From the solutions of system \eqref{eq:sistema_do_grupo}, we obtain the set of solutions of the system:
\begin{equation}
    \label{eq:grupo_2}
    \left\{\begin{array}{rcl}
	x^{{a_0}}y^{{b_0}}&=1\\
	x^{{c_0}}y^{{d_0}}&=1
    \end{array}\right. 
\end{equation}

Note that if $r = 1$, then the solution set is the same as obtained in the previous step. However, we will generalize this.

Define $p_1 = a_1 d_1 - b_1 c_1$. Given that $a_0 = a_1 r$ and $c_0 = c_1 r$, and $b_0 = b_1 s$, $d_0 = d_1 s$, it's easy to see that for each pair of integers $i_0$ and $j_0$
\begin{equation}
    \label{eq:solucao_grupo_2}
    (x_{i,{j_0}},y_{i,{i_0}})=\left(\xi_{r{p_1}}^{{p_1}{j_0}-i(m{b_1}+n{d_1})},\xi_{s{p_1}}^{{p_1}{i_0}+i}\right)
\end{equation}
is a solution of system \eqref{eq:grupo_2}. 

Note that system \eqref{eq:grupo_2} has one distinct solution for each $i = 0, 1, ..., p_1 - 1$, $i_0 = 0, 1, ..., s - 1$, and $j_0 = 0, 1, ..., r - 1$, totaling:
\begin{align*}
	|{p_1}rs|
		&=|({a_1}{d_1}-{b_1}{c_1})rs|\\
		&=|{a_1}r{d_1}s-{b_1}s{c_1}r|\\
		&=|{a_0}{d_0}-{b_0}{c_0}|
\end{align*}
solutions.

As a group, this solution set is isomorphic to $\mathbb{Z}_r \oplus \mathbb{Z}_{s p_1}$, being generated by the elements $\left(\xi_{rp_1}^{p_1 + sl}, \xi_{sp_1}^s\right)$ of order $rp_1$, and $\left(1, \xi_s\right)$ of order $s$.

In this step, we obtain the solution set of $G$, from system \eqref{eq:grupooriginal}, using the solutions of system \eqref{eq:grupo_2}. Consider the set of pairs $(x_{i, j_0}, y_{i, i_0})$ from Equation \eqref{eq:solucao_grupo_2}. It's easy to see that for each pair of integers $i_1$, $j_1$, the pair
\begin{equation}
    \label{eq:solucao_grupo_3}
    (x_{i,{j_0},{j_1}},y_{i,{i_0},{i_1}})=\left(\xi_{kr{p_1}}^{r{p_1}{j_1}+{p_1}{j_0}-i(m{b_1}+n{d_1})},\xi_{ks{p_1}}^{s{p_1}{i_1}+{p_1}{i_0}+i}\right)
\end{equation}
is an element of $G$. 

Observe that for each triple $(i, i_0, j_0)$, Equation \eqref{eq:solucao_grupo_3} yields $k^2$ elements of $G$, as $i_1, j_1 = 0, 1, ..., k - 1$.

Letting $l = -m b_1 - n d_1$, we can write the solution from Equation \eqref{eq:solucao_grupo_3} as
\begin{equation}
    \label{eq:solucao_grupo_32}
    \left(\xi_{kr}^{r{j_1}+{j_0}},\xi_{ks}^{s{i_1}+{i_0}}\right)
    \left(\xi_{kr{p_1}}^{il},\xi_{ks{p_1}}^{i}\right)
\end{equation}

Considering the ranges for $i_0, i_1, j_0, j_1$ that yield distinct elements of $G$, we see that the expressions $r j_1 + j_0$ and $s i_1 + i_0$ cover the sets $\{0, ..., kr-1\}$ and $\{0, ..., ks-1\}$, respectively. Thus, we can represent the elements of Equation \eqref{eq:solucao_grupo_32} simply as:
\begin{equation}
    \label{eq:solucao_grupo_33}
    (x_{i,{j_0}},y_{i,{i_0}})=
        \left(\xi_{kr}^{{j_0}},\xi_{ks}^{{i_0}}\right)
        \left(\xi_{kr{p_1}}^{il},\xi_{ks{p_1}}^{i}\right)
\end{equation}
where $i = 0, ..., p_1 - 1$, $i_0 = 0, ..., ks - 1$, and $j_0 = 0, ..., kr - 1$.

Taking all possible values for $i, i_0, j_0$, we find that:
\begin{align*}
	|G|
	&=|{p_1}kskr|\\
	&=|({a_1}{d_1}-{b_1}{c_1})kskr|\\
	&=|(({a_1}rk)({d_1}sk)-({b_1}sk)({c_1}rk))|\\
	&=|ad-bc|
\end{align*}

Regarding the structure of $G$, we observe that it is generated by the elements
$$z_1=\left(\xi_{kr},1\right)
\left(\xi_{kr{p_1}}^{sl},\xi_{ks{p_1}}^{s}\right)$$
which has order $krp_1$ (obtained with $(i, i_0, j_0) = (s, 0, 1)$) and
$$z_2=\left(\xi_{kr}^{r},\xi_{ks}\right)
\left(1,1\right)$$
which has order $ks$ (obtained with $(i, i_0, j_0) = (0, 1, r)$). Considering these generators, we conclude:
$$G\cong \mathbb Z_{kr{p_1}}\oplus \mathbb Z_{ks}$$

We can summarize this entire discussion in the following result:

\begin{theorem}
\label{teo:estrutura_grupo_isotropia}
Let $a, b, c, d$ be positive integers such that $ad - bc \neq 0$. Let $G$ be the subgroup of $\Bbbk^* \oplus \Bbbk^*$ consisting of elements $(x, y)$ such that $x^a y^b = 1$ and $x^c y^d = 1$. Let $k = \gcd(a, b, c, d)$, $r = \gcd(a/k, c/k)$, $s = \gcd(b/k, d/k)$, and $p = (ad - bc) / (k^2 r s)$. Then

$$G \cong \mathbb{Z}_{krp} \oplus \mathbb{Z}_{ks}$$

In particular, $|G| = |ad - bc|$.
\end{theorem}

\begin{remark}
    Observe that when $\Bbbk$ is not algebraically closed, we obtain $$G \subseteq \mathbb{Z}_{krp} \oplus \mathbb{Z}_{ks}.$$
\end{remark}

\section{Some observations using the Intersection Index:}

In this section, we identify each element of the isotropy of a derivation $\delta=\ad_w$, where $w=x^ay^b+x^cy^d$, in $A=\Bbbk_q[x,y]$ with a point on the algebraic torus $\Bbbk^* \times \Bbbk^*$. Suppose further that if $q^n=1$, then $(a,b), (c,d)\notin n\mathbb Z\times n\mathbb Z$: this tells us that the case of Proposition \ref{Proq} does not apply. Using geometric properties, we assume in this section that $\Bbbk$ is algebraically closed and of characteristic zero.

Also, note that when we homogenize the curves $F:x^ay^b-1=0$ and $G:x^cy^d-1=0$, the remaining intersection points in the projective $\mathbb{P}^2$ must be only one of the following: $(1:0:0)$ or $(0:1:0)$.

Denoting by $(F,G)_P$ the multiplicity, or index, of intersection of two plane curves $F$ and $G$ at the point $P \in F \cap G$.

\begin{example} Using the same context as in Example 5., that is, considering $\bar{F}:x^3y-z^4$, $\bar{G}:x^2y^2-z^4$, then:$$(\bar{F},\bar{G})_P=(x^3y-z^4, \bar{G})_P=(x^2y(x-y),\bar{G})_P=$$$$=2(x,\bar{G})_P+(y,\bar{G})_P+(x-y,\bar{G})_P=2.4+4+4.$$

We obtain exactly as obtained in Example 5: $4$ points in the affine part; more precisely, $\mathbb Z_4$. Furthermore, we have that the multiplicity of $(1:0:0)$ is $4$ and the multiplicity of $(0:1:0)$ is $8$.


\end{example}

\begin{remark}
Two curves $x^ay^b=1$ and $x^cy^d=1$, when they have no irreducible component in common, $ad-bc \neq 0$, have an intersection index at the identity, point $(1,1)$, equal to $1$. In fact, the tangent lines are given by $a(x-1)+b(y-1)=0$ and $c(x-1)+d(y-1)=0$ which, by the condition $ad-bc \neq 0$, are not parallel. Note that the same is true for any intersection point $(x_0,y_0)$ in the affine part, whose tangent lines are $a(x-x_0)+b(y-y_0)=0$ and $c(x-x_0)+d(y-y_0)=0$.
\end{remark}

\begin{proposition} \label{prop.curvas}Consider the curves $F: x^ay^b=1$ and $G: x^cy^d=1$ with $ad-bc \neq 0$ and $\gcd(a,b)=\gcd(c,d)=1$. Then, $$(F,G)_{(0:1:0)}=ac+ \min\{bc,ad\}$$

and
$$(F,G)_{(1:0:0)}=bd+\min\{da,bc\}.$$

\end{proposition}
\begin{proof}

At the point $(0:1:0)$, we obtain the equations $\bar{F}:x^a=z^{a+b}$ and $\bar{G}:x^c=z^{c+d}$. Taking the following Puiseux parameterization of $\bar{F}:$ given by $\varphi(t)=(t^{a+b},t^a)$. We obtain, by \cite[Theorem 14]{brie}, that the multiplicity at the point $(0:1:0)$ is the order of $0$ in $\varphi^{*}(\bar{G})$. Since $\varphi^{*}(\bar{G})=t^{c(a+b)}-t^{a(c+d)}=t^{ac}(t^{cb}-t^{ad})$. Thus, the multiplicity of $(0:1:0)$ is given by $ac+ \min\{bc,ad\}$. At the point $(1:0:0)$, we obtain, in the same way as the previous point, that the multiplicity is given $bd+\min\{da,bc\}$. \end{proof}

\begin{remark} \label{obs.curvas}
If $\gcd(a,b)=d_1>1$, we can write $$x^ay^b-1={(x^{a'}y^{b'})^{d_1}}-1= \prod_{k=0}^{d_1-1}(x^{a'}y^{b'}-\xi^k),$$ with $a'=\frac{a}{d_1}$, $b'=\frac{b}{d_1}$ and $\xi$ is the $d_1$-th root of unity. Therefore, each factor defines an irreducible affine curve, since $\Bbbk$ is algebraically closed and by irreducibility in $\Bbbk[x^{\pm 1},y^{\pm 1}]$: thus, the curve $F: x^ay^b-1=0$ has $d_1$ distinct irreducible components. Likewise, the projective curve $\bar{F}:x^ay^b-z^{a+b}=0$ given by the homogenization of $x^ay^b-1=0$ factors into $d_1$ distinct irreducible components: $$\bar{F}=\prod_{k=0}^{d_1-1}(x^{a'}y^{b'}-\xi^kz^{a'+b'}).$$

If $\gcd(c,d)=d_2>1$, we can similarly write the curves $G:x^cy^d-1=0$ and $\bar{G}:x^cy^d-z^{c+d}=0$ as: $$G=\prod_{k=0}^{d_2-1}(x^{c'}y^{d'}-\zeta^k),$$ $$\bar{G}=\prod_{k=0}^{d_2-1}(x^{c'}y^{d'}-\zeta^kz^{c'+d'}),$$ both with $d_2$ distinct irreducible components.

Note that all components of $\bar{F}$ and $\bar{G}$ pass through both projective points $(0:1:0)$ and $(1:0:0)$. And so, similarly the previous Proposition and using the additivity of the intersection index, we have:

$$(F,G)_{(0:1:0)}=d_1d_2(a'c'+ \min\{b'c',a'd'\})$$

and
$$(F,G)_{(1:0:0)}=d_1d_2(b'd'+\min\{b'c',a'd'\}).$$
\end{remark}

\begin{example}
Given the curves $F: x^2y^4=1$ and $G: x^3y^9=1$. Note that we have a total of $72$ points counted with multiplicity, of which $18-12=6$ distinct points are in the affine part. Let us observe that the remaining $66$ points are divided between $(0:1:0)$ and $(1:0:0)$. In fact, take the branch $xy^2=1$ of $F$ and the branch $xy^3=1$ of $G$; we thus obtain: multiplicity $3$ in $(0:1:0)$ and multiplicity $8$ in $(1:0:0)$.
\end{example}

As a consequence of Proposition \ref{prop.curvas} and Observation \ref{obs.curvas}, using Bezout's Theorem, we could show in another way part of the result obtained in Theorem \ref{teo:estrutura_grupo_isotropia}: that is, there are, in the affine part, exactly $|ad-bc|$ distinct points in the intersection of both curves.

\section*{Acknowledgement}

The authors would like to thank  Ivan Pan (CMAT-UdelaR), Paula Lomp (FCUP), Christian Lomp (FCUP) and Gabriel Fazoli (IMPA) for their literature suggestions and helpful comments. We are also grateful for the referee’s insightful comments and valuable suggestions. This work was funded by the Rio Grande do Sul Research Foundation (FAPERGS).

\vspace{5mm}
\small

\ttfamily ADRIANO DE SANTANA, UNIVERSIDADE TECNOLÓGICA FEDERAL DO PARANÁ, UTFPR TOLEDO/PR, BRASIL.

\textit{E-mail address:} adrianosantana@utfpr.edu.br 
\vspace{3mm}

\ttfamily RENE BALTAZAR, UNIVERSIDADE FEDERAL DO RIO GRANDE, FURG, SANTO ANTÔNIO DA PATRULHA/RS, BRASIL.

\textit{E-mail address:} renebaltazar.furg@gmail.com

\vspace{3mm}
\ttfamily ROBSON VINCIGUERRA, UNIVERSIDADE TECNOLÓGICA FEDERAL DO PARANÁ, UTFPR TOLEDO/PR, BRASIL.

\textit{E-mail address:} robsonwv@gmail.com
\vspace{3mm}

\ttfamily WILIAN DE ARAUJO, UNIVERSIDADE TECNOLÓGICA FEDERAL DO PARANÁ, UTFPR TOLEDO/PR, BRASIL.

\textit{E-mail address:} wilianmat@yahoo.com.br

\end{document}